\documentclass[10pt]{article}

\usepackage{amssymb}
\usepackage{amsthm}
\usepackage{amsmath}
\usepackage{enumerate}
\usepackage{hyperref}

\newtheorem{thm}{Theorem}

\newcommand*{\di}{\, \mathrm{d} }

\title{A simple wide range approximation of symmetric binomial distributions}

\author{Tam\'as Szabados\footnote{Address:
Department of Mathematics, Budapest University of Technology and Economics, Muegyetem rkp. 3, H
ep, 5 em, Budapest, 1521, Hungary, e-mail: szabados@math.bme.hu, telephone: +36 1 463 1111/ext. 5905, fax: +36 1 463 1677} \\
Budapest University of Technology and Economics}

\begin{document}

\maketitle

\begin{abstract}
The paper gives a wide range, uniform, local approximation of symmetric binomial distribution. The result clearly shows how one has to modify the the classical de Moivre--Laplace normal approximation in order to give an estimate at the tail as well to minimize the relative error.
\end{abstract}


The topic of this paper is a wide range, uniform, local approximation of symmetric binomial distribution, an extension of the classical de Moivre--Laplace theorem in the symmetric case \cite[Ch. VII]{Fel68}. In other words, I would like to approximate individual binomial probabilities not only in a classical neighborhood of the center, but at the tail as well. The result clearly shows how one has to modify the normal approximation in order to give a wide range estimate to minimize the relative error. The method will be somewhat similar to the ones applied by \cite{Fel45, Lit69, McK89} or in the proof of Tusn\'ady's lemma, see e.g. \cite{Mas02}; but my task is much simpler than those. This simplification makes the proof short, transparent and natural. Moreover, the result is non-asymptotic, that is, it gives explicit, nearly optimal, upper and lower bounds for the relative error with a finite $n$. Thus I hope that it can be used in both applications and teaching.

Let $(X_r)_{r \ge 1}$ be a sequence of independent, identically distributed random steps with $\mathbb{P}(X_r = \pm 1) = \frac12$ and $S_{\ell} = \sum_{r=1}^{\ell} X_r$ $(\ell \ge 1)$, $S_0 = 0$, be the corresponding simple, symmetric random walk. Then
\begin{equation}\label{eq:binprob}
\mathbb{P}(S_{\ell} = j) = \binom{\ell}{\frac{\ell + j}{2}} 2^{-2 \ell} \qquad (|j| \le \ell).
\end{equation}
Here we use the convention that the above binomial coefficient is zero whenever $\ell + j$ is not divisible by $2$. Since $\mathbb{P}(S_{\ell} = j) = \mathbb{P}(S_{\ell} = -j)$ for any $j$, it is enough to consider only the case $j \ge 0$ whenever it is convenient.

First we consider the case when $\ell = 2n$, even. Let us introduce the notation
\[
a_{k,n} := \mathbb{P}(S_{2n} = 2k) = \binom{2n}{n+k} 2^{-2n} \qquad (|k| \le n) .
\]
Also, introduce the notation
\begin{equation}\label{eq:bkn}
b_{k,n} := n \left\{\left(1 + \frac{k+\frac12}{n}\right) \log\left(1 + \frac{k}{n}\right) + \left(1 - \frac{k-\frac12}{n}\right) \log\left(1 - \frac{k}{n}\right)\right\}
\end{equation}
when $n \ge 1$ and $|k| < n$, and
\begin{equation}\label{eq:bnn}
b_{\pm n,n} := \left(2n + \frac12\right) \log 2 - \frac12 \log(2 \pi n) \qquad (n \ge 1).
\end{equation}

\begin{thm}\label{th:moivre}
(a) For any $n \ge 1$ and $|k| \le n$, we have
\begin{equation}\label{eq:exp_appr_up}
a_{k,n} \le \frac{1}{\sqrt{\pi n}} e^{-b_{k,n}} .
\end{equation}

(b) For any $n \ge 1$ and $|k| \le rn$, $r \in (0,1)$, we have
\begin{equation}\label{eq:exp_appr_down}
a_{k,n} > \frac{1}{\sqrt{\pi n}} e^{-b_{k,n}} \, \exp\left(- \frac{1}{7n} - \frac{r^4}{3(1-r^2)^2 \, n}\right) .
\end{equation}

(c) In accordance with the classical de Moivre--Laplace normal approximation, for $n \ge 3$ and $|k| \le n^{\frac23}$ one has
\begin{equation}\label{eq:class_Moivre_Lapl}
1 - 2 n^{-\frac13} < \frac{a_{k,n}}{\frac{1}{\sqrt{\pi n}} e^{-\frac{k^2}{n}}} < 1 + 2 n^{-\frac13} .
\end{equation}
\end{thm}
\begin{proof}
As usual, the first step is to estimate the central term
\[
a_{0,n} = \binom{2n}{n} 2^{-2n} = \frac{(2n)!}{(n!)^2} 2^{-2n}.
\]
By Stirling's formula, see e.g. \cite[p.~54]{Fel68}, we have:
\begin{equation}\label{eq:Stirling}
\sqrt{2 \pi n} \left(\frac{n}{e}\right)^n e^{\frac{1}{12 n + 1}} < n! < \sqrt{2 \pi n} \left(\frac{n}{e}\right)^n e^{\frac{1}{12 n}} \qquad (n \ge 1).
\end{equation}
Thus after simplification we get
\begin{equation}\label{eq:central}
\frac{1}{\sqrt{\pi n}} e^{-\frac{1}{7n}} < a_{0,n} < \frac{1}{\sqrt{\pi n}} e^{-\frac{1}{9n}}  \qquad (n \ge 1).
\end{equation}

Second, also by the standard way, for $1 \le k \le n$,
\[
a_{k,n} = a_{0,n} \frac{n (n-1) \cdots (n-k+1)}{(n+1)(n+2) \cdots (n+k)}
= a_{0,n} \frac{\left(1-\frac{1}{n}\right) \left(1-\frac{2}{n}\right) \cdots  \left(1-\frac{k-1}{n}\right)}{\left(1+\frac{1}{n}\right) \left(1+\frac{2}{n}\right) \cdots \left(1+\frac{k}{n}\right)} .
\]

So it follows that
\begin{multline}\label{eq:noncentral}
\log a_{k,n} = \log a_{0,n} - \log\left( 1+\frac{k}{n} \right) - 2 \sum_{j=1}^{k-1} \frac12 \log \frac{1+\frac{j}{n}}{1-\frac{j}{n}}  \\
= \log a_{0,n} - \log\left( 1+\frac{k}{n} \right) - 2 \sum_{j=1}^{k-1} \tanh^{-1}\left(\frac{j}{n}\right) .
\end{multline}
To approximate the sum here, let us introduce the integral
\begin{equation}\label{eq:int_appr}
I(x) := \int_0^x \tanh^{-1}(t) \di t = \frac12 (1+x) \log(1+x) + \frac12 (1-x) \log(1-x)
\end{equation}
for $|x| < 1$, and its approximation by a trapezoidal sum
\begin{equation}\label{eq:trapez_appr}
T_{k,n} := \frac{1}{n} \left\{\frac12 \tanh^{-1}(0) + \sum_{j=1}^{k-1} \tanh^{-1}\left(\frac{j}{n}\right) + \frac12 \tanh^{-1}\left(\frac{k}{n}\right) \right\} ,
\end{equation}
where $0 \le k < n$. It is well-known that the error of the trapezoidal formula for a function $f \in C^2([a, b])$ is
\[
T_n(f) - \int_a^b f(t) \di t  = \frac{(b-a)^3}{12 n^2} f''(x), \qquad x \in [a, b].
\]
Since $(\tanh^{-1})''(x) = 2x(1-x^2)^{-2}$, we obtain that
\begin{equation}\label{eq:appr_error}
0 \le T_{k,n} - I\left(\frac{k}{n}\right) \le \frac{r^4}{6(1-r^2)^2 \, n^2}, \quad \text{when} \quad 0 \le k \le r n , \quad r \in (0,1).
\end{equation}

Let us combine formulas (\ref{eq:central})--(\ref{eq:appr_error}):
\[
\log a_{k,n} = \log a_{0,n} - 2 n T_{k,n} - \frac12 \log\left(1 + \frac{k}{n}\right) - \frac12 \log\left(1 - \frac{k}{n}\right),
\]
thus
\begin{equation}\label{eq:appr_err1}
\log \frac{1}{\sqrt{\pi n}} - b_{k,n} - \frac{1}{7n} - \frac{r^4}{3(1-r^2)^2 \, n}  < \log a_{k,n} < \log \frac{1}{\sqrt{\pi n}} - b_{k,n} - \frac{1}{9n},
\end{equation}
where $0 \le k \le r n$ and
\begin{equation}\label{eq:bkn0}
b_{k,n} = 2n I\left(\frac{k}{n}\right) + \frac12 \log\left(1 - \frac{k^2}{n^2}\right) .
\end{equation}
Clearly, (\ref{eq:bkn0}) is the same as (\ref{eq:bkn}). Thus (\ref{eq:appr_err1}) proves (a) and (b) of the theorem.

Let us see now, using Taylor expansions, a series expansion of $b_{k,n}$ when $|k| < n$. First, for $|x| < 1$,
\[
I(x) = \frac{x^2}{1 \cdot 2} + \frac{x^4}{3 \cdot 4} + \frac{x^6}{5 \cdot 6} + \frac{x^8}{7 \cdot 8} + \cdots .
\]
Second, also for $|x| < 1$,
\[
\log(1-x^2) = -x^2 - \frac{x^4}{2} - \frac{x^6}{3} - \frac{x^8}{4} - \cdots  .
\]
So by (\ref{eq:bkn0}), for $|k| < n$ we have a convergent series for $b_{k,n}$:
\begin{multline}\label{eq:series_bkn}
b_{k,n} = \frac{k^2}{n} \left(1-\frac{1}{2n}\right) + \frac{k^4}{2n^3} \left(\frac{1}{3}-\frac{1}{2n}\right) + \frac{k^6}{3n^5} \left(\frac{1}{5}-\frac{1}{2n}\right) + \cdots \\
= \sum_{j=1}^{\infty} \left(\frac{k}{n}\right)^{2j} \frac{1}{j} \left(\frac{n}{2j-1}-\frac{1}{2}\right) .
\end{multline}

By (\ref{eq:series_bkn}), for $n \ge 3$ and $|k| \le n^{\frac23}$ we get that
\begin{multline*}
\left|b_{k,n} - \frac{k^2}{n} \right|  \le  \frac{n^{-\frac23}}{2} + \sum_{j=2}^{\infty} n^{-\frac23 j} \frac{1}{j} \left(\frac{n}{2j-1} + \frac12\right) \\
\le \frac{n^{-\frac23}}{2} + \frac{n}{2} \left(\frac{1}{3} + \frac16\right) \sum_{j=2}^{\infty} n^{-\frac23 j} \le n^{-\frac13} .
\end{multline*}
Thus by (\ref{eq:exp_appr_up}), for $n \ge 3$ and $|k| \le n^{\frac23}$,
\begin{equation}\label{eq:center_up}
a_{k,n} \le \frac{1}{\sqrt{\pi n}} e^{-\frac{k^2}{n}} \, e^{n^{-1/3}} < \frac{1}{\sqrt{\pi n}} e^{-\frac{k^2}{n}} \, \left(1 + 2n^{-\frac13}\right) .
\end{equation}

Similarly, by (\ref{eq:exp_appr_down}), for $n \ge 3$ and $|k| \le n^{\frac23}$,
\begin{multline}\label{eq:center_down}
a_{k,n} > \frac{1}{\sqrt{\pi n}} e^{-\frac{k^2}{n}} \, \exp\left(- \frac{1}{7n} - \frac{n^{-\frac73}}{3(1-n^{-\frac23})^2} - n^{-\frac13}\right) \\
\ge \frac{1}{\sqrt{\pi n}} e^{-\frac{k^2}{n}} \, e^{-1.21 n^{-1/3}} > \frac{1}{\sqrt{\pi n}} e^{-\frac{k^2}{n}} \, \left(1 - 2n^{-1/3}\right) .
\end{multline}
(\ref{eq:center_up}) and (\ref{eq:center_down}) together proves (c) of the theorem.
\end{proof}

The main moral of Theorem 1 is that the exponent $k^2/n$ in the exponent of the normal approximation is only a first approximation of the series (\ref{eq:series_bkn}). In (c) the bound $|k| \le n^{\frac23}$ was chosen somewhat arbitrarily. It is clear from the series (\ref{eq:series_bkn}) that the bound should be $o(n^{\frac34})$. The above given bound was picked because it seemed to be satisfactory for usual applications with large deviation and gave a nice relative error bound $2 n^{-\frac13}$.

It is not difficult to extend the previous results to the odd-valued case of the symmetric binomial probabilities
\[
a^*_{k,n} := \mathbb{P}(S_{2n-1} = 2k-1) = \binom{2n-1}{n+k-1}2^{-2n+1}  \qquad (-n+1 \le k \le n) .
\]
Define
\begin{equation}\label{eq:bkn_odd}
b^*_{k,n} := n \left\{\left(1 + \frac{k-\frac12}{n}\right) \log\left(1 + \frac{k}{n}\right) + \left(1 - \frac{k-\frac12}{n}\right) \log\left(1 - \frac{k}{n}\right)\right\}
\end{equation}
for $n \ge 1$ and $|k| < n$, and
\begin{equation}\label{eq:bnn_odd}
b^*_{n,n} := \left(2n - \frac12\right) \log 2 - \frac12 \log(2 \pi n) \qquad (n \ge 1).
\end{equation}

\begin{thm}\label{th:moivre_odd}
(a) For any $n \ge 1$ and $-n+1 \le k \le n$, we have
\begin{equation}\label{eq:exp_appr_up_odd}
a^*_{k,n} < \frac{1}{\sqrt{\pi n}} e^{-b^*_{k,n}} \, \exp\left(\frac{2}{3n}\right).
\end{equation}

(b) For any $n \ge 1$ and $|k| \le rn$, $r \in (0,1)$, we have
\begin{equation}\label{eq:exp_appr_odd}
a^*_{k,n} > \frac{1}{\sqrt{\pi n}} e^{-b^*_{k,n}} \, \exp\left(- \frac{1}{n} - \frac{r^4}{3(1-r^2)^2 \, n}\right)  .
\end{equation}

(c) In accordance with the classical de Moivre--Laplace normal approximation, for $n \ge 4$ and $|k| \le n^{\frac23}$ one has
\begin{equation}\label{eq:class_Moivre_Lapl_odd}
1 - 3 n^{-\frac13} < \frac{a^*_{k,n}}{\frac{1}{\sqrt{\pi n}} e^{-\frac{k^2}{n}}} < 1 + 6 n^{-\frac13} .
\end{equation}
\end{thm}
\begin{proof}
Since this proof is very similar to the previous one, several details are omitted. First, by Stirling's formula (\ref{eq:Stirling}), after simplifications we get that
\begin{equation}\label{eq:central_odd}
\frac{1}{\sqrt{\pi n}} e^{-\frac{1}{n}} < a^*_{0,n} < \frac{1}{\sqrt{\pi n}} e^{\frac{2}{3n}}  \qquad (n \ge 1).
\end{equation}
Second, similarly to (\ref{eq:appr_err1}), for $n \ge 1$ and $|k| \le rn$, $r \in (0,1)$, we obtain that
\begin{equation}\label{eq:appr_err2}
\log \frac{1}{\sqrt{\pi n}} - b^*_{k,n} - \frac{1}{n} - \frac{r^4}{3(1-r^2)^2 \, n}  < \log a^*_{k,n} < \log \frac{1}{\sqrt{\pi n}} - b^*_{k,n} + \frac{2}{3n},
\end{equation}
where
\begin{equation}\label{eq:bkn_odd2}
b^*_{k,n} := 2n I\left(\frac{k}{n}\right) - \tanh^{-1}\left(\frac{k}{n}\right) \\
= \sum_{j=1}^{\infty} \left(\frac{k}{n}\right)^{2j-1} \frac{1}{2j-1} \left(\frac{k}{j}-1\right) .
\end{equation}
(\ref{eq:bkn_odd2}) clearly agrees with (\ref{eq:bkn_odd}). (\ref{eq:appr_err2}) proves (a) and (b).

By  the series in (\ref{eq:bkn_odd2}), for $n \ge 4$ and $|k| \le n^{\frac23}$ we get that
\[
\left|b^*_{k,n} - \frac{k^2}{n} \right|  < 2 n^{-\frac13} .
\]
This and (\ref{eq:appr_err2}) imply (c).
\end{proof}


\end{document}